\documentclass[11pt]{amsart}
\usepackage{mathptmx}
\usepackage{amssymb}
\usepackage{enumerate}
\usepackage{graphicx}
\usepackage{pinlabel}

\catcode`\@=11

\long\def\@savemarbox#1#2{\global\setbox#1\vtop{\hsize\marginparwidth 
  \@parboxrestore\tiny\raggedright #2}}
\newcommand\lref[1]{\ref{#1}%
\@ifundefined{r@DisplaY #1}{}{ (#1)}}

\newcommand\fakelabel[2]{\@bsphack\if@filesw {\let\thepage\relax
   \newcommand\protect{\noexpand\noexpand\noexpand}%
\xdef\@gtempa{\write\@auxout{\string
      \newlabel{#1}{{#2}{\thepage}}}}}\@gtempa
   \if@nobreak \ifvmode\nobreak\fi\fi\fi\@esphack}

\def\SL@margintext#1{{\showlabelsetlabel{\tiny\{\SL@prlabelname{#1}\}}}}
\catcode`\@=12

\def\Empty{}
\newcommand\oplabel[1]{
  \def\OpArg{#1} \ifx \OpArg\Empty {} \else
        \label{#1}
  \fi}
\newtheorem{theoremSt}{Theorem}[section]

\newtheorem{exampleSt}[theoremSt]{Example}
\newtheorem{exerciseSt}[theoremSt]{Exercise}

\newcommand\MakeStEnv[1]{
  \newenvironment{#1}[1]{
  \begin{#1St} \oplabel{##1}%
  \global\def\CrntSt{\thetheoremSt}%
}{ 
  \end{#1St} }
  \newenvironment{#1+}[1]{
  \begin{#1St} \label{##1}%
  \label{DisplaY ##1}%
  \global\def\CrntSt{\thetheoremSt}%
  \def\Labl{##1}\ifx\Labl\Empty{} \else {\em (\Labl)\,}\fi%
}{ 
  \end{#1St} }
}
\MakeStEnv{theorem}
\MakeStEnv{corollary}
\MakeStEnv{proposition}
\MakeStEnv{lemma}
\MakeStEnv{definition}
\MakeStEnv{conjecture}
\MakeStEnv{problem}
\MakeStEnv{question}

\newlength{\saveu}

\newenvironment{pf*}[1]{%
 \begin{proof}[#1]%
}{ 
 \end{proof}
}

\newcommand{\finishproof}[1]{ 
  \def\FPArg{#1}
  \ifx\FPArg\Empty
        \newcommand\FPArg{\CrntSt}  \fi
  \smallbreak\noindent\makebox[\textwidth]{\hfill\fbox{\FPArg}}
  \medbreak\noindent
}

\newcommand\FF{{\mathcal F}}

\newcommand\LL{{\mathcal L}}
\newcommand\MM{{\mathcal M}}

\newcommand\PP{{\mathcal P}}

\newcommand\PMF{{\PP\kern-2pt\MM\FF}}

\newcommand\PML{{\PP\kern-2pt\MM\LL}}

\newcommand\half{{\textstyle{\frac12}}}

\newcommand\Mod{\operatorname{Mod}}

\newcommand\bbR{{\mathord{\text{I\kern-2pt R}}}}        
\newcommand\bbH{{\mathord{\text{I\kern-2pt H}}}}        

\newcommand\PSL[1]{\text{PSL}_{#1}}

\newcommand\SL[1]{\text{SL}_{#1}}

\newcommand\bigrightarrow[1]{\hbox to #1{\rightarrowfill}}
\newcommand\bigleftarrow[1]{\hbox to #1{\leftarrowfill}}

\newcommand\semidir{\mathrel{\hbox{\vrule depth-.03ex height1.1ex\kern-0.15em$\times$}}}

\newcommand{\diam}{\operatorname{diam}}

\numberwithin{equation}{section}

\newcommand{\inj}{\operatorname{inj}}

\newcommand{\fsubd}{\mathrel{{\scriptstyle\searrow}\kern-1ex^d\kern0.5ex}}
\newcommand{\bsubd}{\mathrel{{\scriptstyle\swarrow}\kern-1.6ex^d\kern0.8ex}}
\newcommand{\fsubeq}{\mathrel{\raise-.7ex\hbox{$\overset{\searrow}{=}$}}}
\newcommand{\bsubeq}{\mathrel{\raise-.7ex\hbox{$\overset{\swarrow}{=}$}}}

\newcommand{\tsh}[1]{\left\{\kern-.9ex\left\{#1\right\}\kern-.9ex\right\}}

\newcommand\Teich{{\mathcal T}}

\newcommand\interior{{\rm int}}

\newcommand{\bb}{\mathbb}

\newcommand{\cx}{{\bb C}}
\renewcommand{\half}{{\bb H}}
\newcommand{\integers}{{\bb Z}}

\newcommand{\ratls}{{\bb Q}}
\newcommand{\reals}{{\bb R}}

\renewcommand{\bold}[1]{\medskip \noindent {\bf \boldmath #1
                        }\nopagebreak[4]}

\newcommand{\bdry}{\partial}

\newcommand{\compos}{\circ}

\newcommand{\st}{\; : \;}         

\newcommand{\zed}{\integers}

\newcommand{\AH}{\operatorname{AH}}
\newcommand{\area}{\operatorname{area}}

\newcommand{\core}{\operatorname{core}}

\renewcommand{\diam}{\operatorname{diam}}

\renewcommand{\inj}{\operatorname{inj}}

\renewcommand{\interior}{\operatorname{int}}

\renewcommand{\Mod}{\operatorname{Mod}}

\renewcommand{\PSL}{\operatorname{PSL}}

\renewcommand{\SL}{\operatorname{SL}}

\renewcommand{\Teich}{\operatorname{Teich}}

\newcommand{\vol}{\operatorname{vol}}

\newcommand{\cC}{{\calC}}

\newcommand{\cK}{{\calK}}

\newcommand{\calC}{{\mathcal C}}

\newcommand{\calK}{{\mathcal K}}

\newcommand{\calM}{{\mathcal M}}

\newcommand{\calV}{{\mathcal V}}
\newcommand{\calW}{{\mathcal W}}

\renewcommand{\hbar}{\bar{{\mathbb H}}^3}

\renewcommand{\core}{C}
\newcommand{\bilip}{\operatorname{bilip}}

\def\SL{\rm{SL}}

\def\core{\operatorname{core}}
\def\sys{\operatorname{sys}}
\def\depth{\operatorname{depth}}

\DeclareGraphicsExtensions{.pdf}

\begin{document}
\title{Inflexibility, Weil-Petersson  distance,
  and volumes of fibered 3-manifolds}
\author[Brock and  Bromberg]{Jeffrey  Brock and Kenneth Bromberg}
\thanks{J. Brock was partially supported by NSF grant DMS-1207572. K. Bromberg was partially supported by NSF grant 
DMS-1207873}

\begin{abstract} A recent preprint of S. Kojima and G. McShane
  \cite{Kojima:Mcshane:volumes} observes a beautiful explicit
  connection between Teichm\"uller translation distance and hyperbolic
  volume. It relies on a key estimate which we supply here: using
  geometric inflexibility of hyperbolic 3-manifolds, we show that for
  $S$ a closed surface, and $\psi \in \Mod(S)$ pseudo-Anosov, the
  double iteration $Q(\psi^{-n}(X),\psi^n(X))$ has convex core volume
  differing from $2n \vol(M_\psi)$ by a uniform additive constant,
  where $M_\psi$ is the hyperbolic mapping torus for $\psi$.  We
  combine this estimate with work of Schlenker, and a branched covering
  argument to obtain an explicit lower bound on {\em Weil-Petersson
    translation distance} of a pseudo-Anosov $\psi \in \Mod(S)$ for
  general compact $S$ of genus $g$ with $n$ boundary components: we
  have
$$   \vol(M_\psi)
\le 3 \sqrt{\pi/2(2g - 2 +n)} \, \| \psi \|_{\rm WP}
.$$ This gives the first explicit estimates on the Weil-Petersson
systoles of moduli space, of
the minimal distance between nodal surfaces in the completion of
Teichm\"uller space, and explicit lower bounds to the Weil-Petersson
diameter of the moduli space via \cite{Cavendish:Parlier:diameter}.  In the process, we recover the estimates of
\cite{Kojima:Mcshane:volumes} on Teichm\"uller translation distance
via a Cauchy-Schwarz estimate (see \cite{Linch:wp}).
\end{abstract}
\date{\today}
\maketitle

\section{Introduction}
Let $S$ be a closed surface of genus $g >1$.
Let $\psi: S \to S$ be a pseudo-Anosov element of $\Mod(S)$, $Q_n =
Q(\psi^{-n}(X), \psi^n (X))$ quasi-Fuchsian simultaneous uniformizations, and $M_\psi$ the
hyperbolic mapping torus for $\psi$. Let $\core(Q_n)$ denote the
convex core of $Q_n$.  We will prove the following
theorem.
\begin{theorem}{main}
The quantity
$$|\vol(\core(Q_n)) - 2n\vol(M_\psi)|$$
is uniformly bounded.
\end{theorem}
The possibility of such a result was suggested in
\cite[\S1]{Brock:3ms1}. It gives an alternative, more direct proof of the main
result of that paper comparing hyperbolic volume of $M_\psi$ and {\em
  Weil-Petersson translation distance} of $\psi$ as a direct corollary
of a similar comparison in the quasi-Fuchsian case \cite{Brock:wp}.  A
recent preprint of S. Kojima and G. McShane shows how this suggestion
can be used to give sharper bounds between volume of $M_\psi$ and {\em
  normalized entropy}, or the translation distance in the
Teichm\"uller metric of $\psi$.

In addition to supplying a proof of Theorem~\ref{main}, we focus here on
volume implications for the Weil-Petersson metric on Teichm\"uller
space.  Indeed, by analyzing the {\em renormalized volume}, rather
than the convex core volume, Jean-Marc Schlenker improved the upper
bound in \cite{Brock:3ms1}, for closed $S$ with 
genus at least $2$.
\begin{theorem}{improved wp bound} {\rm (Schlenker)} Let $S$ be a
  closed surface of genus $g>1$ and let $X$, $Y$ lie in $\Teich(S)$.  There is a
  constant $K_S>0$ so that  
$$\vol(\core(Q(X,Y))) \le 3 \sqrt{ \pi(g-1)} \, d_{\rm
  WP}(X,Y) + K_S.$$ 
\end{theorem}
\noindent (See  \cite{Schlenker:volume}).

The  {\em Weil-Petersson translation length} of $\psi$ as an automorphism of
$\Teich(S)$ is defined by taking the infimum
$$\| \psi \|_{\rm WP} = \inf_{X \in \Teich(S)} d_{\rm WP}(X , \psi(X)).$$ 
Daskalopoulos and Wentworth \cite{Daskalopoulos:Wentworth:mcg} showed
this infimum is realized by $X$ in $\Teich(S)$ when $\psi$ is pseudo-Anosov.

Combining Theorem~\ref{main}, Theorem~\ref{improved wp bound}, and a
covering trick, we obtain the following Theorem.
\begin{theorem}{translation bound}
  Let $S$ be a compact surface with genus $g$ and $n$ boundary
  components $\chi(S)<0$ 
 and let $\psi \in
  \Mod(S)$ be pseudo-Anosov.  Then we have
$$\vol(M_\psi) \le  3 \sqrt{ \frac{\pi}{2} (2g-2+n)}  \,  \| \psi \|_{\rm WP}. $$
\end{theorem}
The case when $S$ is closed readily follows from
  Theorem~\ref{main} and Theorem~\ref{improved wp bound}.
When $S$ has boundary, a branched covering argument allows us to
recover the estimates from the closed case; we defer the proof to
section~\ref{proof}.  


When $S$ has boundary, the Teichm\"uller
space $\Teich(S)$ parametrizes finite area marked hyperbolic structures on
$\interior(S)$ up to marking preserving isometry.  We take
$\area(S)$ to denote the Poincar\'e
area of any $X \in \Teich(S)$, namely 
$$\area(S) = 2\pi(2g -2 +n).$$
Then given $X$, $Y \in \Teich(S)$  one may consider the {\em
  normalized Weil-Petersson distance} 
$$d_{\rm WP^*}(X,Y) = \frac{d_{\rm WP}(X,Y)}{\sqrt{\area(S)}}.$$
Passage to finite covers of $S$ yields an isometry of normalized
Weil-Petersson metrics, as is the case with the Teichm\"uller metric.

Then by an application of the Cauchy-Schwarz inequality (see
\cite{Linch:wp}), we have for each $X$, $Y$ in $\Teich(S)$ the bound
$$d_{\rm WP^*}(X,Y) \le   d_T(X,Y)$$
from which we conclude 
$$\| \psi \|_{\rm WP^*} \le  \| \psi \|_T,$$
where $\|\psi \|_{\rm WP^*}$ denotes the translation distance of
$\psi$ in the normalized Weil-Petersson metric.
As it follows from Theorem~\ref{translation bound} that
$$\vol(M_\psi) \le \frac{3}{2} \area(S) \, \| \psi \|_{\rm WP^*}
\le \frac{3}{2} \area(S) \, \| \psi \|_T,$$ 
we recover the
Theorem of \cite{Kojima:Mcshane:volumes} concerning volumes and
Teichm\"uller translation distance for arbitrary compact $S$.  

We note that the study of normalized entropy and dilatation has seen
considerable interest of late, note in particular the papers of
\cite{Agol:Leininger:Margalit:pseudo}, and
\cite{Farb:Leininger:Margalit:volume} which have greatly improved our
understanding of fibered 3-manifolds of low dilatation.  The work of
\cite{Kojima:Mcshane:volumes} has been particularly important here,
giving a new proof of the finiteness theorem of
\cite{Farb:Leininger:Margalit:volume}, that all mapping tori
arising from pseudo-Anosov monodromy of {\em bounded dilatation} arise
from Dehn filling of those from a finite list.  
We remark that an analogous Theorem where a bound on the normalized
Weil-Petersson translation distance replaces a bound on the dilataion
is immediate from \cite{Brock:3ms1}, together with J\o rgensen's
Theorem \cite{Thurston:book:GTTM}.

 We will focus our attention
primarily on implications for Teichm\"uller space with the
Weil-Petersson metric.

\bold{Weil-Petersson geometry.}  With the abundance of connections
between Weil-Petersson geometry and 3-dimensional hyperbolic volume,
one might wonder about a more direct connection.  Indeed, in
\cite{Manin:Marcolli} Manin and
Marcolli raise the expectation of an exact formula relating the two,
but in joint work of the first author with Juan Souto
\cite{Brock:Souto:exact} it is shown that there is no continuous
function $f \colon \reals \to \reals$ so that $f(\vol(M_\psi)) =
\|\psi\|_{\rm WP}$ (and likewise for the quasi-Fuchsian case).

Nevertheless, for a given $S$, the first author and Yair Minsky show
the following further similarity with the distribution of lengths of
hyperbolic volumes \cite{Brock:Minsky:spectrum}.
\begin{theorem}{spectrum} \cite{Brock:Minsky:spectrum} {\sc (Length Spectrum)}
The extended Weil-Petersson geodesic length spectrum of $\calM(S)$  is
a well ordered subset of $\reals$, with order type $\omega^\omega$.
\end{theorem}
Here, the extended length spectrum refers to the set of lengths of
closed geodesics together with lengths of {\em extended mapping
  classes}, automorphisms of a Teichm\"uller-Coxeter complex
introduced by Yamada, where Dehn-twist iterations can take infinite
powers. Such limiting elements behave as {\em billiard paths} on the
moduli space with the Weil-Petersson completion, intersecting the
compactification divisor with equal angle of incidence and reflection
(see \cite{Wolpert:compl}, \cite{Yamada:coxeter}).

It is natural to speculate regarding the value of the bottom of this
spectrum, or the {\em systole} of the moduli space $\calM(S)$ with the
Weil-Petersson metric: Theorem~\ref{translation bound} gives the first
explicit estimates on the value of the systole. It was shown by Gabai,
Meyerhoff and Milley \cite{Gabai:Meyerhoff:Milley}, that the smallest
volume closed orientable hyperbolic 3-manifold is the Weeks manifold
$\calW$, obtained by $(5,2)$ and $(5,1)$ Dehn surgeries on the
Whitehead link.  An explicit formula for its volume is given by
$$\vol(\calW) = \frac{3 \cdot 23^{3/2} \zeta_k(2)}{4 \pi^4}.$$

Applying Theorem~\ref{translation bound}, we 
conclude the following lower bound on the Weil-Petersson systole of
$\calM(S)$ for $S$ a closed surface.

\begin{theorem}{wp systole closed}{\sc (Weil-Petersson Systole -
    Closed Case)}
Let $S$ be a closed surface with genus $g>1$, and let $\gamma$ be the shortest
 closed Weil-Petersson geodesic in the moduli space $\calM(S)$.
Then we have
$$\frac{\vol(\calW)}{3 \sqrt{\pi(g-1)}} \le 
\ell_{\rm   WP}(\gamma).$$
\end{theorem}
We remark that a recent result of Agol,
Leininger and Margalit \cite{Agol:Leininger:Margalit:pseudo}
provides an upper bound:
$$\ell_{\rm   WP}(\gamma)
\le 
\frac{2 \sqrt{\pi} \log(\frac{3+\sqrt{5}}{2})}{\sqrt{(g-1)}}.$$

Similarly, Cao and Meyerhoff \cite{Cao:Meyerhoff:minvol} show that the smallest volume orientable cusped
hyperbolic 3-manifold is the figure eight knot complement which has
volume $ 2 \calV_3$ where $\calV_3$ is the volume of
the regular ideal hyperbolic tetrahedron.
An application of this bound yields a similar result for the
Weil-Petersson systole of the moduli space of punctured surfaces.
\begin{theorem}{wp systole punctured}{ \sc (Weil-Petersson Systole -
    Punctured Case)}
Let $S$ be a surface of genus $g$ with $n>0$ boundary components
and
$\chi(S) < 0$, and let $\gamma$ be the shortest Weil-Petersson 
geodesic in the moduli space $\calM(S)$.
Then we have
$$\frac{ 2 \calV_3} {3 \sqrt{ \frac{\pi}{2}  (2g -2
      +n)}}
\le \ell_{\rm WP}(\gamma).$$
\end{theorem}
\noindent Known upper bounds require a more involved discussion, which
we omit here.

\bold{The Weil-Petersson inradius of Teichm\"uller space.} 
It is remarkable that even to estimate the distance between nodal
surfaces at infinity in the Weil-Petersson metric has been an elusive
problem.  Theorem~\ref{translation bound} provides the first
explicit means by which to do this, through a limiting process
involving Dehn-twist iterates about a longitude-meridian pair
$(\alpha,\beta)$ on the punctured torus.  

Specifically, letting $S$ be the one-holed torus, we identify the
upper-half-plane $\half^2$ with $\Teich(S)$. Then $\Mod(S) =
\SL_2(\zed)$ acts by isometries, and we consider the family
$$\psi_n = \tau_\alpha^n \compos \tau_\beta^{-n}$$ 
of composed $n$-fold Dehn-twists about simple closed curves $\alpha$ and
$\beta$ on $S$ with $i(\alpha,\beta) = 1$ on $S$. Up to conjugation,
we have
$$\psi_n = 
\left[ 
\begin{array}{cc}
1 & n \\
0 & 1
\end{array}
\right]
\left[
\begin{array}{cc}
1 & 0 \\
n & 1
\end{array}
\right]
$$
in $\SL_2(\zed)$.
Then Theorem~\ref{translation bound} gives 
$$\vol(M_{\psi_n}) \le 
3
\sqrt{\frac{\pi}{2} } 
\|\psi_n\|_{\rm WP}.$$
The left hand side converges to $2 \calV_8$ or twice the volume of the
ideal hyperbolic octahedron, while the right hand side converges to
twice the Weil-Petersson length of the imaginary axis, in the upper
half plane $\half^2$.

Then we obtain:
\begin{theorem}{farey edge}{\sc (Weil-Petersson Inradius)}
Let $S= S_{1,1}$ be the one-holed torus.  The Weil-Petersson length of
the imaginary axis $I$ satisfies 
$$\frac{1}{3}\sqrt{ \frac{2}{\pi}} \calV_8 \le 
\ell_{\rm  WP}(I) \le
2 \sqrt{ 30}\, \pi^\frac{3}{4}.$$
\end{theorem}
The upper bound arises from estimates on the systole $\sys(X)$ of a hyperbolic
surface $X$, together with Wolpert's upper bound of $\sqrt{2 \pi
  \sys(X)}$ \cite[Cor 4.10]{Wolpert:behavior}
on the distance from $X$ to a nodal surface where a curve of length
$\sys(X)$ on $X$ is pinched to a cusp (see \cite{Cavendish:Parlier:diameter}).

The axis $I$ is isometric to each edge $e$ of
the {\em Farey graph} $\mathbb{F} = \Gamma (I)$, where $\Gamma = \SL_2(\zed)$, joining pairs of
rationals $(\frac{p}{q}, \frac{r}{s})$ with
$$
\left|
\begin{array}{cc}
p & q \\
r & s
\end{array}
\right| = 1$$
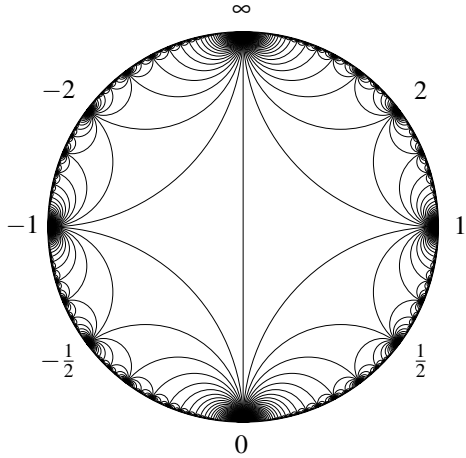
\begin{figure}[htb]
\labellist
\small\hair 2pt
 \pinlabel {$\infty$}  at 305 670
 \pinlabel {$-1$}  at 30 400
 \pinlabel {$1$}  at 580 400
 \pinlabel {$0$}  at 305 125
 \pinlabel {$\frac{1}{2}$}  at 530 230
 \pinlabel {$-\frac{1}{2}$}  at 75 230
 \pinlabel {$2$}  at 530 570
 \pinlabel {$-2$}  at 75 570
\endlabellist
\centering
\includegraphics[scale=.3]{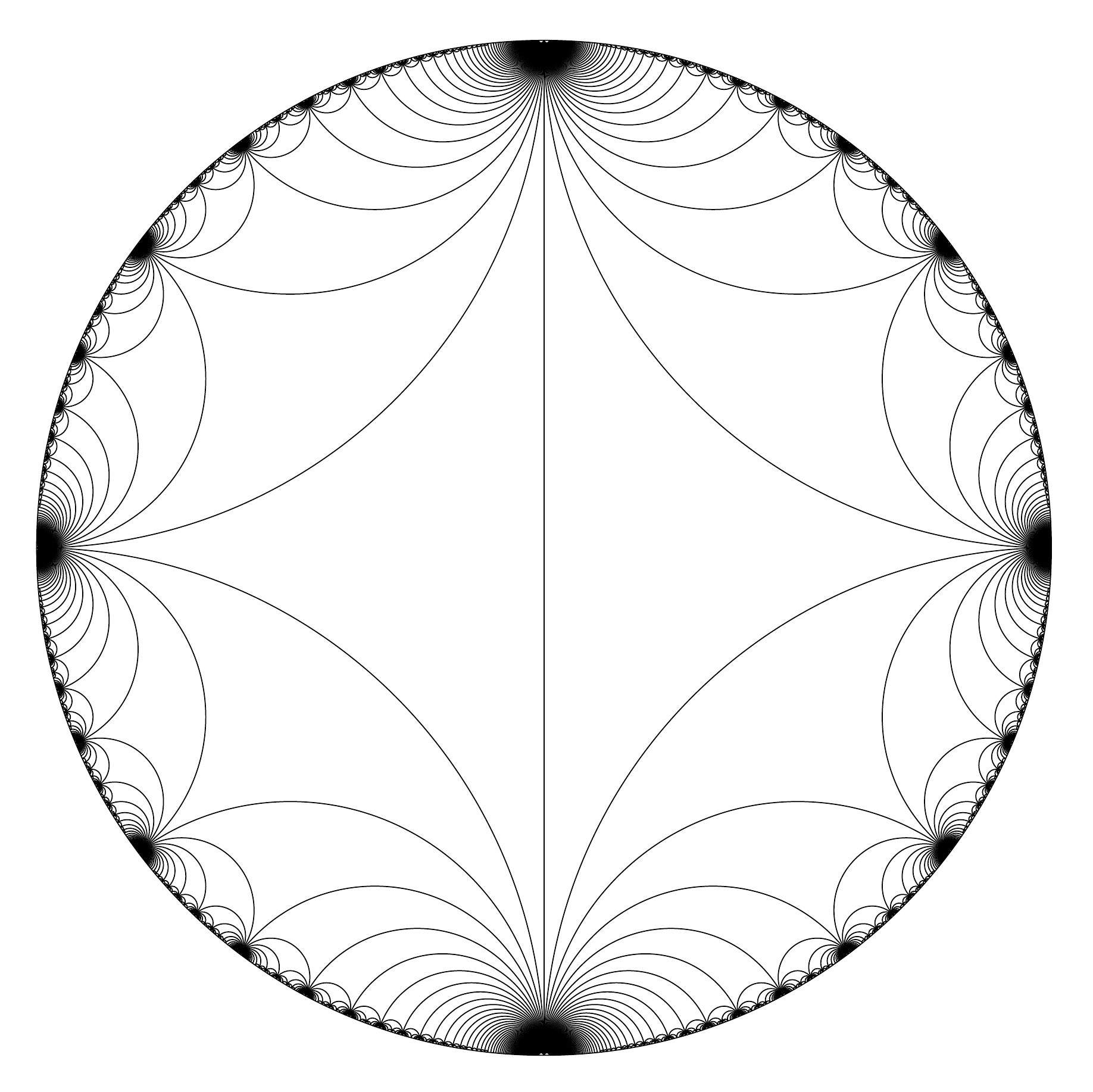}
\caption{The Farey graph under stereographic projection to
   $\Delta$.}
\label{figure:farey}
\end{figure}
or the extended distance
$$d_{\overline{\rm WP}}\left(\frac{p}{q},\frac{r}{s}\right) = \
d_{\overline{\rm WP}}(0,\infty)$$
 between closest points in the {\em completion} 
$$\overline{\Teich(S)}  =  \half^2\cup (\ratls \cup \infty)$$
of the Teichm\"uller space of the punctured torus with the
Weil-Petersson metric (see Figure~\ref{figure:farey}).

\bold{Weil-Petersson diameter of Moduli space.} The length $\ell_{\rm
  WP}(I)$ is instrumental in the estimation of the Weil-Petersson
diameter of moduli space \cite{Cavendish:Parlier:diameter}.  Noting
that the Weil-Petersson length of $I$ in the Teichm\"uller space of
the four-holed sphere is twice its length in the Teichm\"uller space
of the one-holed torus, we may combine Theorem~\ref{farey edge} with
results of \cite{Cavendish:Parlier:diameter} relating this length to
the diameter of moduli space to obtain the following explicit
estimates:
\begin{theorem}{wp diameter}
  Let $S = S_{g,n}$ have $\chi(S) <0$, and let $\calM_{g,n} =
  \calM(S_{g,n})$, the moduli space of genus $g$ Riemann surfaces with
  $n$ punctures. Then we have the following:
$$\diam_{\rm WP}(\calM_{1,1}) \ge \frac{1}{6}\sqrt{\frac{2}{\pi}}
\calV_8, $$
$$\diam_{\rm WP}(\calM_{0,4}) \ge \frac{1}{3}\sqrt{\frac{2}{\pi}}
\calV_8, $$
and otherwise for $3g-3+n \ge 2$,
$$
\diam_{\rm WP}
(\calM_{g,n}) \ge \frac{1}{3 \sqrt{\pi}} \calV_8 \sqrt{2g +n -4}.
$$
\end{theorem}
\begin{proof}
  The imaginary axis in $\half$ projects 2-to-1 to a geodesic in
  $\calM_{1,1}$ of half its original Weil-Petersson length, which is
  estimated in Theorem~\ref{farey edge}. The Weil-Petersson metric on
  $\calM_{0,4}$ is isometric to twice that of $\calM_{1,1}$. The
  general estimate follows from the totally geodesically embedded
  $\calM_{0,2g+n}$ strata in the completion of $\calM_{g,n}$, as
  observed in \cite[Prop. 5.1]{Cavendish:Parlier:diameter}.
\end{proof}

\noindent We note that dividing by $\area(S)$ gives an explicit, positive
lower bound to the normalized Weil-Petersson diameter 
of moduli spaces $\calM(S)$ that is independent of $S$.

\bold{History.}  The original version of \cite{Kojima:Mcshane:volumes}
relied without proof on a remark in \cite{Brock:3ms1} suggesting a
proof of Theorem~\ref{main} should be possible using the idea of {\em
  geometric inflexibility} from \cite{McMullen:book:RTM} and
\cite{Brock:Bromberg:inflexible}.  The present paper supplies such a
proof, as a means toward employing Schlenker's improvement
\cite[Cor. 1.4]{Schlenker:volume} to the upper bound in
\cite[Thm. 1.2]{Brock:wp} to obtain new explicit estimates on the
Weil-Petersson geometry of Teichm\"uller and moduli space.  After we
presented our arguments in Curt McMullen and Martin Bridgeman's
Informal Seminar at Harvard, a revision to
\cite{Kojima:Mcshane:volumes} presented an independent proof of
Theorem~\ref{main} (as well as version of Theorem~\ref{translation
  bound} restricted to closed surfaces) and McMullen provided a
succinct argument for a slightly weaker version of Theorem~\ref{main}
directly from Thurston's {\em Double Limit Theorem} and the strong
convergence of $Q_n$ to the fiber $Q_\infty$.  His argument appears in
his Seminar Notes available on his webpage, together with other
estimates for $L^p$ metrics on Teichm\"uller space (and alongside
notes from our lecture). We have retained the inflexibility approach
here to illustrate its utility.

\bold{Acknowledgements.}  We are grateful to Greg McShane and
Sadayoshi Kojima for bringing this potential application of
inflexibility to our attention.  We also thank Curt McMullen and
Martin Bridgeman for suggestions and corrections and for giving us the
opportunity to present our arguments in their Informal Seminar, and
McMullen for suggesting the inclusion of the upper bound for
Weil-Petersson systole. We also thank
Chris Leininger for helpful conversations and Jean-Marc Schlenker for
his comments and corrections to estimates in \cite{Schlenker:volume}.
\section{Preliminaries}
We review background for our results.

\bold{Weil-Petersson geometry.}  The above results give new explicit
estimates on the geometry of Teichm\"uller space with the
Weil-Petersson metric.  The Weil-Petersson metric arises from the
hyperbolic $L^2$-norm on the space of quadratic differentials $Q(X)$
on a Riemann surface $X$, given by 
$$ \| \varphi\|^2_{\rm WP} = \int_X \frac{|\varphi|^2}{\rho_X}.$$
Though known to be geodesically convex \cite{Wolpert:Nielsen} it is not
complete \cite{Wolpert:noncompleteness,Chu:noncompleteness}.  It
has negative curvature \cite{Ahlfors:curvature}, but its curvatures
are bounded away neither from $0$ nor negative infinity. In 
\cite{Daskalopoulos:Wentworth:mcg}, 
 Daskalopoulos and
Wentworth showed that a pseudo-Anosov automorphism $\psi \in \Mod(S)$
has an invariant axis along which $\psi$ translates.  A primitive
pseudo-Anosov element $\psi \in \Mod(S)$, therefore, determines a
closed Weil-Petersson geodesic in the moduli space of Riemann surfaces
$\calM(S)$.

\bold{The complex of curves.}  Let $S$ be a compact
surface of genus $g$ with $n$ boundary components.  The complex of curves $\calC(S)$ is a $3g-4$ dimensional
complex, each vertex of which is associated to a simple closed curve
on the surface $S$ up to isotopy, and so that $k$-simplices span
collections of $k+1$ vertices whose associated curves are disjoint.
Masur and Minsky proved $\calC^1(S)$ is a $\delta$-hyperbolic metric
space with the distance $d_{\calC}(.,.)$ given by the edge metric.

Given $S$ there is a an $L_S$ so that for each $X \in \Teich(S)$ there
is a $\gamma \in \calC(S)$ so that $\ell_X(\gamma) < L_S$.  By making
such a choice of $\gamma$ for each $X$ we obtain a coarsely
well-defined projection
$$\pi_\calC \colon \Teich(S) \to \calC^0(S).$$
We refer to the distance between a point $X$ in Teichm\"uller space and a
curve $\gamma$ in $\calC^0(S)$ with the notation:
$$d_\calC(X, \gamma) = d_\calC(\pi_\calC(X),\gamma).$$

\bold{Quasi-Fuchsian manifolds.} Each pair $(X,Y) \in \Teich(S)
\times \Teich(S)$ determines a quasi-Fuchsian {\em simultaneous
  uniformization} $Q(X,Y)$ with $X$ and $Y$ in its conformal boundary.
This is the quotient
$$Q(X,Y) = \half^3 / \rho_{X,Y}(\pi_1(S))$$ of a quasi-Fuchsian
representation of the fundamental group 
$$\rho_{X,Y} \colon \pi_1(S) \to \PSL_2(\cx).$$

The quasi-Fuchsian representations sit as the interior of the space
$\AH(S)$ of all marked hyperbolic 3-manifolds homotopy equivalent to
$S$ up to marking-preserving isometry, with the topology of
convergence on generators of the fundamental group. For more
information, see
\cite{Bromberg:bers,Brock:Bromberg:density,Brock:Canary:Minsky:elc,Agol:tame,Calegari:Gabai:tame}.

A complete hyperbolic 3-manifold $M$, marked by a homotopy equivalence 
$$f \colon S \to M$$ determines a point in $\AH(S)$ up to equivalence -
we denote such a marked hyperbolic 3-manifold by the pair $(f, M)$.
Equipping $M$ with a baseframe $(M,\omega)$ determines a specific
representation $\rho \colon \pi_1(S) \to \PSL_2(\cx)$ and a Kleinian
surface group 
$$\Gamma = \rho(\pi_1(S)).$$

The geometric topology on such based hyperbolic 3-manifolds records
geometric information: a sequence $(M_n , \omega_n )$ converges to
$(M,\omega_\infty)$ if for each $\epsilon, R>0$ there is an $N > 0 $, and for
all $n >N$ we have
embeddings $$\varphi_n \colon (B_R(\omega),\omega) \to
(M_n,\omega_n)$$ 
from the $R$-ball around $\omega$ to $M_n$, whose derivatives send
$\omega_\infty$ to 
$\omega_n$ and whose bi-Lipschitz constants are at most $1+\epsilon$
at all points of $B_R(\omega_\infty)$.  

The convergent sequence $(f_n, M_n) \to (f_\infty, M_\infty)$ in
$\AH(S)$ converges {\em strongly} if there are baseframes $\omega_n$
in $M_n$ and $\omega_\infty \in M_\infty$ so that the resulting
$\rho_n$ converge to the resulting $\rho_\infty$ on generators, and
the manifolds $(M_n,\omega_n)$ converge geometrically to
$(M_\infty,\omega_\infty)$.

\bold{Convex core width.}
Given $M \in \AH(S)$, let $d_M(U,V)$ be the minimal distance between subsets $U$ and $V$ in $M$. We prove the following in \cite{Brock:Bromberg:inflexible}.
\begin{theorem}{distance in manifold}
Given $\epsilon$, $L>0$, there exist $K_1$ and $K_2$ so that if $M \in
\AH(S)$  and
$\alpha$ and $\beta$ in $\calC^0(S)$ have representatives $\alpha^*$
and $\beta^*$ with
$\ell_M(\alpha^*)$ and $\ell_M(\beta^*) $ bounded above by $L$ and
below by $\epsilon$, then
$$d_M(\alpha,\beta) \ge K_1 d_\calC(\alpha, \beta) -K_2$$
\end{theorem}

It is due to Bers that 
$$2 \ell_X(\gamma) \ge \ell_{Q(X,Y)}(\gamma).$$ Thus
Theorem~\ref{distance in manifold} serves to bound from below the
width of the convex core of $Q(X,Y)$ (the distance between its
boundary components) in terms of the curve complex distance.   Such
convex core width estimates will be important to our application of the
inflexibility theory outlined in the next section.  

 

\section{Geometric Inflexibility} 
To prove Theorem~\ref{main}, our key tool will be the {\em inflexibility theorem} of
\cite{Brock:Bromberg:inflexible}.
\begin{theorem}{inflexible} {\sc (Geometric Inflexibility)}
Let $M_0$ and $M_1$ be complete hyperbolic structures on a
$3$-manifold $M$ so that $M_1$ is a $K$-quasi-conformal deformation of
$M_0$, $\pi_1(M)$ is finitely generated, and $M_0$ has no rank-one
cusps.  

There is a volume preserving $K^{3/2}$-bi-Lipschitz diffeomorphism
$$\Phi \colon M_0 \to M_1$$ 
whose pointwise bi-Lipschitz constant satisfies
$$\log \bilip (\Phi, p) \le C_1 e^{-C_2 d(p, M_0 \setminus
  \core(M_0))}$$ for each $p \in M^{\ge \epsilon}$, where $C_1$ and
$C_2$ depend only on $K$, $\epsilon$, and $\area(\bdry \core(M_0))$.
\end{theorem}
The existence of a volume preserving, $K^{3/2}$ bi-Lipschitz diffeomorphism was
established by Reimann \cite{Reimann:visual}, using work of Ahlfors
\cite{Ahlfors:visual} and Thurston \cite{Thurston:book:GTTM}  (see McMullen
\cite{McMullen:book:RTM} for a self-contained account).
  That the bi-Lipschitz constant
decays exponentially fast with depth in the convex core at points in
the thick part follows from comparing $L^2$ and pointwise bounds on
{\em harmonic strain fields} arising from extending a {\em Beltrami 
isotopy} realizing the deformation.  Exponential decay of the $L^2$
norm in the core can be converted to pointwise bounds via mean value
estimates, building on work in the cone-manifold deformation theory of
hyperbolic manifolds due to Hodgson and Kerckhoff
\cite{Hodgson:Kerckhoff:rigidity} and the second author
\cite{Bromberg:bers}.

Inflexibility was used in \cite{Brock:Bromberg:inflexible} to give a
new, self-contained proof of Thurston's {\em Double Limit Theorem},
and the hyperbolization theorem for  closed 3-manifolds that fiber over the
circle with pseudo-Anosov monodromy.
\begin{theorem}{strong double limit}{\rm (Thurston)} Let $S$ be a closed
  hyperbolic surface.  The sequence $Q(\psi^{-n}(X), \psi^n (X))$
  converges algebraically and geometrically to $Q_\infty$, the
  infinite cyclic cover of the mapping torus $M_\psi$ corresponding to
  the fiber $S$.
\end{theorem}
The manifolds 
$$Q_n = Q(\psi^{-n}(X), \psi^n(X))$$ admit volume preserving, uniformly
bi-Lipschitz {\em   Reimann maps} 
$$\phi_n \colon Q_n \to Q_{n+1}$$ as in Theorem~\ref{inflexible}.
The key to obtaining Theorem~\ref{strong double limit} from
Theorem~\ref{inflexible} is an analysis of the growth rate of the
convex core diameter in terms of the curve complex.

We will employ the following key consequence of inflexibility 
\cite[Prop. 9.7]{Brock:Bromberg:inflexible}.
\begin{proposition}{inflexible iteration} Given $\epsilon, R, L,C>0$
  there exist $B, C_1, C_2>0$ such that the following holds. Assume
  that $\cK$ is a subset of $Q_N$ such that $\diam(\cK)<R$,
$\inj_p(\cK) > \epsilon$ for each 
  $p \in \cK$ and $\gamma \in \cC^0(S)$ is represented by a closed
  curve in $\cK$ of length at most $L$ satisfying
$$\min\{d_\calC(\psi^{N+n}(Y), \gamma),
d_\calC(\psi^{-N-n}(X), \gamma\} \ge K_\psi n +B$$ for all $n \geq 0$.
Then we have
$$\log \bilip (\phi_{N+n}, p) \leq C_1 e^{-C_2 n}$$ 
for $p$ in $\phi_{N+n-1} \circ \cdots \circ \phi_{N} (\cK)$
and
$$\frac{C_1}{1 - e^{-C_2}} < C.$$
\end{proposition}
The simple closed curve $\gamma$ serves to control the depth of the
compact set $\calK$ in the convex core of $Q_{N+n}$ as $n \to \infty$
via inflexibility and Theorem~\ref{distance in manifold}: if $\calK$
starts out sufficiently deep, then the geometry freezes around it
quickly enough that Theorem~\ref{distance in manifold} guarantees its
depth grows linearly, resulting in the exponential convergence of the
bi-Lipschitz constant.

\bold{Double Iteration.}  The pseudo-Anosov double iteration $\{Q_n\}$ converges
strongly to the doubly degenerate manifold $Q_\infty$, invariant by
the isometry
$$\Psi \colon Q_\infty \to Q_\infty$$ 
the isometric covering translation for $Q_\infty$ over the mapping
torus $M_\psi$ for $\psi$ (see \cite{Thurston:hype2,
  Cannon:Thurston:peano, McMullen:book:RTM,
  Brock:Bromberg:inflexible}).  
Likewise, McMullen showed the
iteration $Q(X,\psi^n(X))$ also converges strongly to a limit $Q_{X,\psi^\infty}$
in the {\em Bers slice}
$$B_X = \{Q(X,Y) \st Y \in \Teich(S) \}.$$

Each element $\tau \in \Mod(S)$ acts on $\AH(S)$ by {\em remarking},
or precomposition of the representation by the corresponding
automorphism of the fundamental group.  This action is denoted by 
$$ \tau(f,M) = (f \compos \tau^{-1},M)$$

Then by Thurston's {\em Double Limit Theorem}
\cite{Thurston:hype2,Otal:book:fibered,Brock:Bromberg:inflexible}, the remarking of $Q_{X,\psi^\infty}$ by $\psi^{-n}$ produces a sequence
$$\psi^{-n}(Q_{X,\psi^\infty}) = Q_{\psi^{-n}(X),\psi^\infty}$$
converging strongly 
in $\AH(S)$ to $Q_\infty$ (see \cite{McMullen:book:RTM}).

Bonahon's Tameness Theorem \cite{Bonahon:tame} provides
a homeomorphism
$$ F \colon S \times \reals \to Q_\infty$$ equipping the 
limit $Q_\infty$ with a product structure; we assume the isometric
covering transformation
$$\Psi \colon Q_\infty \to Q_\infty$$
in the homotopy class of $\psi$ for the covering $Q_\infty \to M_\psi$
preserves this product structure and acts by integer translation
$\Psi(S,t) = (S,t+1)$ in the second factor.
We denote by $Q_\infty[a,b]$ the subset $F(S \times [a,b])$.  

\begin{proposition}{deep volume}
  Let $\gamma \in \calC^0(S)$ satisfy $\ell_X(\gamma) < L_S$. Then
  there exists $a>0$ and $N_1>0$ so that for each $n>N_1$ the compact
  subset $Q_\infty[-a,a]$ contains $\gamma^*$ and admits a marking
  preserving 2-bi-Lipschitz embedding
$$\varphi_n \colon Q_\infty[-a,a] \to Q_{\psi^{-n}(X),\psi^\infty}.$$ 
\end{proposition}

\begin{proof}
  The Proposition follows from the observation that the geodesic
  representatives of $\psi^n(\gamma)$ lie arbitrarily deep in the
  convex core of $Q_{X,\psi^\infty}$, and the fact that the isometric
  remarkings $\psi^{-n}(Q_{X,\psi^\infty}) = Q_{\psi^{-n}(X),\psi^\infty}$ converge
  to the fiber $Q_\infty$ (see
  \cite[Thm. 3.11]{McMullen:book:RTM}). Choosing an interval $[-a,a]$
  so that $Q_\infty[-a,a]$ contains $\gamma^*$, the marking preserving
  bi-Lipschitz embeddings
$$\varphi_n \colon Q_\infty[-a,a] \to Q_{\psi^{-n}(X),\psi^\infty}$$
are eventually $2$-bi-Lipschitz, giving the desired $N_1$.
\end{proof}
We note that we may argue symmetrically for $Q_{\psi^{-\infty},
  \psi^n(X)}$, the strong limit of $Q(\psi^{-m}(X), \psi^n(X))$ as $m
\to \infty$.

\section{The Proof}
\label{proof}

In this section we give the proof of Theorem~\ref{main}.  
\begin{proof} The proof is a straightforward application of
  Proposition~\ref{inflexible iteration}.  Making an initial choice of
  $N$, we will find, for each $k$, a subset $\calK_k$ of $Q_{N+k}$
  accounting for all but a uniformly bounded amount of the volume of
  the core of $Q_{N+k}$. Fixing $k$, the volume preserving Reimann
  maps $\phi_{N+k+n}$ of Proposition~\ref{inflexible iteration}
  applied to $\calK_k$ produce subsets of $Q_{N+k+n}$ that converge as
  $n \to \infty$ to a subset of $Q_\infty$ within bounded volume of
  $2k$ copies of the fundamental domain for the action of $\psi$. This
  yields the desired comparison.
 
\bold{Step I.} {\em Choose constants.} 
As the input for Proposition~\ref{inflexible iteration}, let $$R > 4\diam
(Q_\infty[-a,a]),$$ 
 take $L > 4L_S$, and 
 fix $\epsilon < \epsilon_\psi/4$, where 
$$\epsilon_\psi = \inj(M_\psi) = \inj (Q_\infty)$$
where $\inj(M) = \inf_{p \in M} \inj_p(M)$.  Finally, taking $C=2$, we
take $B$, $C_1$ and $C_2$ satisfying the conclusion of
Proposition~\ref{inflexible iteration}.
Recall that $\gamma \in \calC^0(S)$ satisfies $\ell_X (\gamma) <
L_S$.  
Then applying \cite[Thm. 8.1]{Brock:Bromberg:inflexible} there is an $N_0>0$ so that 
$$\min \{ d_\calC (\psi^{-N_0-n}(X), \gamma) ,
d_\calC(\psi^{N_0+n}(X),\gamma) \} \ge K_\psi n +B$$ for all $n\ge0$.

\bold{Step II.} {\em Geometric convergence.}
Applying Proposition~\ref{deep volume}, there is an $N_1$ which we may
take so that $N_1 >N_0$ 
so that for each $N> N_1$ there are  $2$-bi-Lipschitz embeddings 
$$\varphi_N^- \colon Q_\infty[-a,a] \to \core( Q_{\psi^{-N}(X),\psi^\infty})$$
$$\varphi_N^+ \colon Q_\infty[-a,a] \to \core( Q_{\psi^{-\infty},\psi^{N}(X)})$$
that are marking preserving.  
Applying strong convergence of 
$$Q(Y,\psi^n(X)) \to Q_{Y,\psi^\infty} \ \ \ \text{and} \ \ \ 
Q(\psi^{-n}(X), Y) \to Q_{\psi^{-\infty},Y},$$ we take $N_2 >N_1$ so that
for each $\delta>0$, $D>0$, and $N> N_2$, we have $k_0$ so that for
$k>k_0$ there are diffeomorphisms
$$\eta^-_{N,k}  \colon Q_{\psi^{-N}(X),\psi^\infty} \to
Q(\psi^{-N}(X),\psi^{N+2k}(X))$$
and
$$\eta^+_{N,k}  \colon Q_{\psi^{-\infty},\psi^N(X)} \to
Q(\psi^{-N-2k}(X),\psi^{N}(X))$$ 
so that $\eta^-_{N,k}$  
has bi-Lipschitz constant satisfying
$\log \bilip(\eta^-_{N,k},p)< \delta$  for all points in the 
$D$-neighborhood of $\varphi^-_{N}(Q_\infty[-a,a])$, and likewise for
$\eta^+_{N,k}$. 

It follows that if we fix $N$ satisfying $N>N_2$ for the remainder of
the argument, we have
that the images $\varphi^-_N(Q_\infty[-a,a])$ 
and $\varphi^+_N(Q_\infty[-a,a])$ determine
product regions in $\core(Q_{\psi^{-N}(X),\psi^\infty})$ and $\core(Q_{\psi^{-\infty}, \psi^{N}(X)})$
whose complements contain one product region of volume bounded by
$\calV>0$.  

Noting that the action by $\psi^{\pm k}$ on $\AH(S)$ gives
$$\psi^{-k}(Q(\psi^{-N} (X), \psi^{N+2k}(X))) = Q_{N+k} = \psi^k ( Q(\psi^{-N -2k}(X), \psi^N(X))),$$ we let, for each
$k>0$, the subsets $\calK^-_k$ and $\calK^+_k$ in $Q_{N+k}$ 
be given by
$$\psi^{-k} ( \eta^-_{N,k} \compos \varphi^-_{N}(Q_\infty[-a,a])) \
\ \ \text{and} \ \ \ \psi^k(\eta^+_{N,k} \compos
\varphi^+_{N}(Q_\infty[-a,a]))$$ by following the embeddings of
$Q_\infty[-a,a]$ from geometric convergence with the isometric
remarkings $\psi^{-k}$ and $\psi^k$.  Then geometric convergence
implies that for each $k>k_0$ the component of $Q_{N+k} \setminus
\calK_k^-$ facing $\psi^{-N -k}(X)$ has intesection with the convex
core bounded by $2\calV$ for $k$ large, and likewise for $\calK_k^+$.

\bold{Step III.} {\em Apply Inflexiblility (Proposition~\ref{inflexible iteration}).}
We take $B$ as in Proposition~\ref{inflexible iteration} given the above
choices for $\epsilon$, $L$, $R$ and $C$.

For our choice of $N$, we know $\calK_k^+$ and $\calK_k^-$ each have diameter
at most $R$, injectivity radius at least $\epsilon$, and as $L =
4L_S$, $\calK_k^+$ and $\calK_k^-$ contain representatives $\gamma_k^-$ of
$\psi^{-k}(\gamma)$ and $\gamma_k^+$ of $\psi^{k}(\gamma)$ of length
less than $L$.  As $B$ is chosen as in the output of
Proposition~\ref{inflexible iteration} and $N$ is chosen as above we have
$$\min \{ d_\calC(\psi^{-N-k-n} (X), \gamma_k^-) ,
d_\calC(\psi^{N+k+n}(X) ,\gamma_k^-) \} \ge K_\psi n + B$$ is satisfied
for all $n\ge0$ and likewise for $\gamma_k^+$.

Let
$\phi_N \colon Q_N \to Q_{N+1}$ denote the (marking preserving)
Reimann map furnished by Proposition~\ref{inflexible iteration}. Then the composition of Reimann maps
$$\Phi_{n} = \phi_{N+n} \compos \ldots \compos \phi_N \colon Q_N \to
Q_{N+n+1}$$ is globally volume  preserving.

Furthermore, since $\calK_k^+$ and $\calK_k^-$ each satisfy the
hypotheses of Proposition~\ref{inflexible iteration} the compositions
are uniformly bi-Lipschitz as $n \to \infty$. It follows from
Arzela-Ascoli that we may extract a limit
limit $\Phi_\infty$ on $\calK_k^+$ and that $\Phi_\infty$ sends $\gamma_k^+$ to a curve of length
at most $8L$ and likewise for $\calK_k^-$ and $\gamma_k^-$. Since
$\Phi_\infty$ is $2$-bi-Lipschitz on $\calK_k^-$ and $\calK_k^+$, it
follows that $\Phi_\infty(\calK_k^-)$ has diameter $8R$, and contains a
representative of $\psi^{-k}(\gamma)$ of length $8L_S$ and likewise
for $\Phi_\infty(\calK_k^+)$ and $\psi^k(\gamma)$. 
There is thus a $d>0$ depending only on $R$ and $L_S$ and
$\epsilon_\psi$ so that we have
$$\Phi_\infty (\calK_k^-) \subset Q_\infty[-k-d,-k+d] \ \ \ \text{and}
\ \ \ 
\Phi_\infty(\calK_k^+) \subset Q_\infty[k-d,k+d].$$

Furthermore, if we take $k$ large enough, we may apply
Theorem~\ref{distance in manifold} to conclude that 
$$d_{Q_{N+k+n}}(\gamma_k^-,\gamma_k^+) >
16R$$ which ensures that $\Phi_{n+k}(\calK_k^-)$ and $
\Phi_{n+k}(\calK_k^+)$ are disjoint for all $n \ge 0$.  The complement
$Q_{N+k} \setminus \calK_k^- \cup \calK_k^+$ contains one subset
$O_{N+k}$ with compact closure `between' the product regions
$\calK_k^-$ and $\calK_k^+$.

Letting $$\calK_k = \calK_k^- \cup O_{N+k} \cup
\calK_k^+,$$
the images
$\Phi_{k+n}(\calK_k)$ satisfy
$$\vol(\calK_k) = \vol(\Phi_{k+n}(\calK_k))$$ since 
$\Phi_{k+n}$ is the composition of volume preserving maps.

But strong convergence of $Q_{N+k+n}$ to $Q_\infty$ as $n\to
\infty$  guarantees that
for large $n$ there are nearly isometric marking-preserving embeddings
$$G_n \colon Q_\infty[-k-d, k+d] \to Q_{N+k+n}$$ that are surjective
onto $\Phi_{k+n}(\calK_k)$ for $n$ sufficiently large.

We conclude that $$(2k-2d)\vol(M_\psi)  \le
\vol(\Phi_\infty(\calK_k)) \le (2k+2d) \vol(M_\psi)$$
and that $$\vol(\core(Q_{N+k})) - 4\calV \le \vol(\calK_k) \le
\vol(\core(Q_{N+k})$$  for all $k$ sufficiently large.
Thus we conclude $$|\vol(\core(Q_{N+k})) - 2(N+k) \vol(M_\psi)| < 2(d+N)\vol(M_\psi) +
4 \calV$$
completing the proof.
\end{proof}

To complete the proof of Theorem~\ref{translation bound}, we conclude
the section by addressing the case when $S$ has boundary.

\begin{proof}[Proof of Theorem~\ref{translation bound}.] 
  We now complete the proof of Theorem~\ref{translation bound}.  It
  remains to treat the case when $S$ has boundary.  We thank Ian Agol
  for suggesting such an argument applies in the setting of the
  Teichm\"uller metric; we employ a similar line of reasoning for the
  Weil-Petersson metric, recovering the Teichm\"uller case as a
  consequence.

We note the following: by Ahlfors Lemma \cite{Ahlfors:simplicial}, for a surface $S = S_{g,n}$ with genus $g>1$ and
$n>0$ boundary components, the natural forgetful map $$\Teich(S_{g,n})
\to \Teich(S_{g,0})$$ obtained by filling in the $n$ punctures on a
surface $X \in \Teich(S_{g,n})$ is a contraction of Poincar\'e
metrics and thus of Weil-Petersson
metrics (see e.g. \cite{Schumacher:Trapani:conical}). Assuming an even
number of punctures, we may branch at the punctures to obtain
degree-$k$ covers $\tilde{S}_k$.

Recall that the normalized Weil-Petersson distance $d_{\rm
  WP^*}(.,.)$, obtained 
by taking 
$$ d_{\rm WP^*}(.,.) = \frac{d_{\rm WP}(.,.)}{ \sqrt{\area(S)}},$$ is invariant under 
the passage to finite covers: lifting to finite covers induces an isometry
of normalized Weil-Petersson metrics.  Given $\psi$ pseudo-Anosov, let
$\|\psi\|_{\rm WP^*}$ denote its translation length in the
  normalized Weil-Petersson metric.

Letting $\psi \in \Mod(S)$, then, we let
$\tilde{\psi}_k$ denote the lift to $\Mod(\tilde{S}_k)$, and
$\hat{\psi}_k \in \Mod(\hat{S}_k)$ obtained by filling in the
punctures of $\tilde{S}_k$ to obtain $\hat{S}_k$.  

Then we have
$$\|\psi\|_{\rm WP^*} = \| \tilde{\psi}_k \|_{\rm WP^*} \ge
 C_k \cdot  \| \hat{\psi}_k \|_{\rm WP^*} 
$$
where $C_k = \sqrt{ \area(\hat{S}_k)/\area(\tilde{S}_k)} \to 1$ as $k \to \infty$.
Applying Theorem~\ref{translation bound} in the closed case we obtain,
$$\| \psi \|_{\rm WP^*} \ge  C_k \frac{2}{3}
\frac{\vol(M_{\hat{\psi}_k})}{\area(\hat{S}_k)}.$$

As $M_{\hat{\psi}_k}$ admits an order-$k$ isometry
corresponding to the $k$-fold branched covering, it covers a fibered
orbifold with $n$ order-$k$ orbifold loci, which converges
geometrically to the fibered 3-manifold $M_{\psi}$ as $k \to
\infty$.  Likewise, $\hat{S}_k$ covers an orbifold with $n$ cone points
with cone-angle $2\pi/k$, whose area is $\area(\hat{S}_k)/k$, 
which converges to $\area(S)$ as
$k\to \infty$.

Thus, dividing the top and the bottom by $k$, the right hand side of
the inequality tends to
$$\frac{2}{3}
\frac{\vol(M_{\psi})}{\area(S)}$$ as
$k \to \infty$, and the estimate holds.

Since any $S = S_{g,n}$ with $n>0$ is finitely covered by $S_{g',n'}$
with $g'>1$ and $n'$ even, the proof is complete.
 \end{proof}

\section{Applications}
We note the following applications to the Weil-Petersson geometry of
Teichm\"uller space.

When $\alpha$ and $\beta$ are a longitude and meridian pair on the
punctured torus, the estimate of Theorem~\ref{farey edge} gives a lower bound 
$$\frac{\calV_8}{3 \sqrt{\pi/2}} \le \ell_{\rm WP}(e)$$
to any edge $e$ in the Farey graph $\mathbb{F}$.
We remark that this estimate has implications for effective
combinatorial models for $\Teich(S)$.

In particular, the main result of \cite{Brock:wp} guarantees the
existence of $K_1$, $K_2$ depending only on $S$ so that
$$\frac{d_P(P_1,P_2)}{K_1} - K_2 \le d_{\rm WP}(N(P_1) ,N(P_2)) \le
K_1 d_P(P_1,P_2) + K_2.$$
Here, the distance $d_P$ is taken in the {\em pants graph} $P(S)$
whose vertices are associated to pants decompositions of $S$ and whose
edges are associated to prescribed elementary moves (see
\cite{Brock:wp}, or \cite{Brock:pdwp} for an expository account) and
$N(P_i)$ denotes the unique maximally noded Riemann surface in the boundary of
Teichm\"uller space for which the curves in $P_i$ have been pinched to
cusps.  
To date, effective estimates on $K_1$ and $K_2$ have been elusive.

Theorem~\ref{translation bound} gives the following estimate in
the case of the punctured torus $S$, on which each pants decomposition
is represented by a single non-peripheral simple closed curve.
\begin{theorem}{pants}
Let $S$ be a one-holed torus and let $\alpha$ and $\beta$ denote
essential simple closed curves on $S$.
If $d_P(\alpha,\beta) = 1$ then 
$$ \frac{\calV_8}{3 \sqrt{\pi/2}} \le d_{\overline{\rm WP}}(N(\alpha),
N(\beta))
\le 2 \sqrt{ 30}\, \pi^\frac{3}{4}
$$
and if $d_P(\alpha,\beta) >1$ then we have
$$ \frac{ \calV_3}{3 \sqrt{\pi/2}}  
d_P(\alpha,\beta)
\le 
d_{\overline{\rm WP}}(N(\alpha), N(\beta)) 
\le
2 \sqrt{ 30}\, \pi^\frac{3}{4}
d_P(\alpha,\beta)$$
\end{theorem}

\begin{proof}
  The space $\Teich(S)$ is naturally the unit disk $\Delta$, and edges
  of the usual Farey graph are geodesics in the Weil-Petersson (as
  well as Teichm\"uller) metric.  Once $d_P(\alpha,\beta)$ is at least $2$,
  the completed Weil-Petersson geodesic $g$ in $\overline{\Teich(S)}$
  joining $N(\alpha)$ to $N(\beta)$ joins the endpoints of a {\em
    Farey sequence}, or a sequence $e_1,\ldots,e_n$ in $\mathbb{F}$
  that joins $\alpha$ to $\beta$.  Each pair of successive edges $e_i$
  and $e_{i+1}$  determines a {\em pivot}, where they meet, and
  emanating from each pivot is a {\em bisector} $b_i$ that meets the
  opposite edge of the ideal triangle determined by $e_i$ and
  $e_{i+1}$ perpindicularly (see Figure~\ref{figure:bisector}).
\begin{figure}[htb]
\labellist
\small\hair 2pt
 \pinlabel {$e_1$}  at 76 298
 \pinlabel {$e_2$}  at 239 299
 \pinlabel {$e_3$}  at 438 289
 \pinlabel {$b_1$}  at 185 10
 \pinlabel {$b_2$}  at 355 450
 \pinlabel {$\alpha$}  at 0 50
 \pinlabel {$\beta$}  at 500 471
\endlabellist
\centering
\includegraphics[scale=.4]{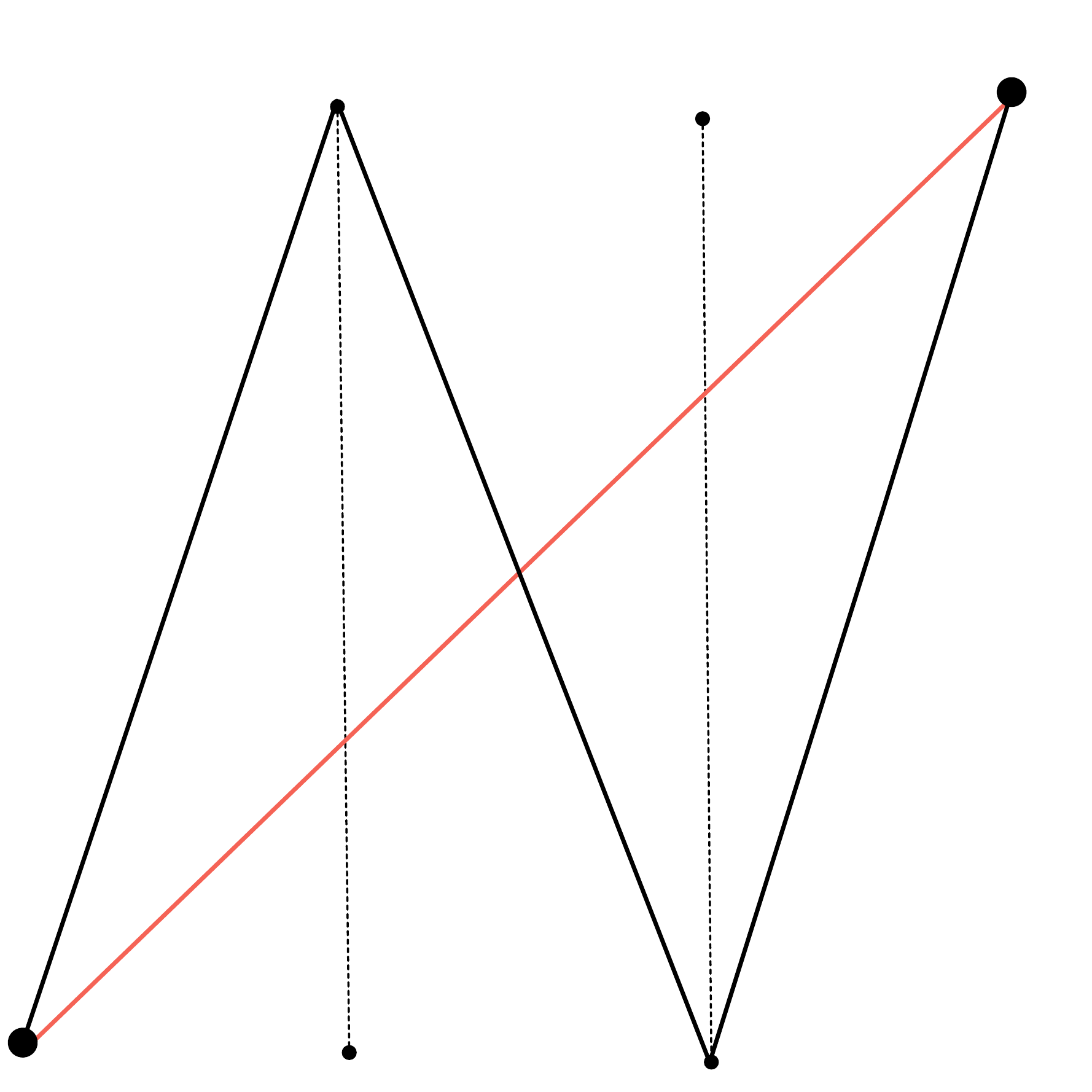}
\caption{Bisectors of the Farey pivots have separation at least
  $\|\psi_{\rm fig8}\|_{\rm WP}/2$  }
\label{figure:bisector}
\end{figure}
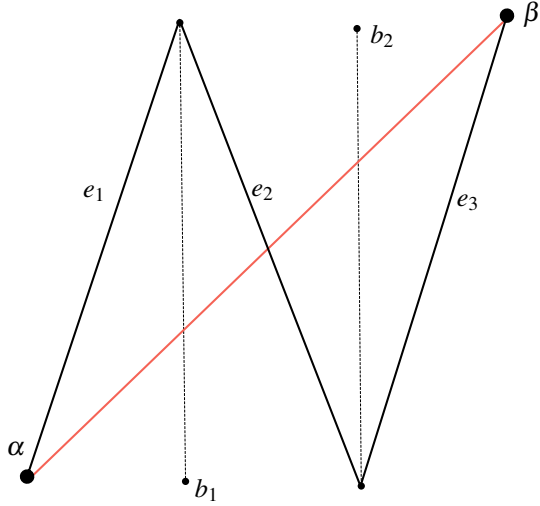
For a Farey sequence that determines {\em exactly one} Farey triangle
per pivot, these bisectors are perpindicular to the axis determined by
a conjugate of the monodromy of the figure-$8$ knot complement, the
mapping class
$$\psi_{\rm fig8} = 
\left[ 
\begin{array}{cc}
2 & 1 \\
1 & 1
\end{array}
\right]
=
\left[ 
\begin{array}{cc}
1 & 1 \\
0 & 1
\end{array}
\right]
\left[
\begin{array}{cc}
1 & 0 \\
1 & 1
\end{array}
\right]
,$$
and the intersections occur every half-period along the axis.

Thus, in this minimal case, successive bisectors have separation at
least half the translation distance of $\psi_{\rm fig8}$ or 
at least $$ \frac{\calV_3}{ 3 \sqrt{\pi/2}}$$
by Theorem~\ref{wp systole punctured}.  When there are more triangles
per pivot, the successive bisectors are further apart.  Thus, the
bisectors determined by the Farey sequence have at least the
separation of the minimal case, as do the initial and terminal
vertices $\alpha$ and $\beta$ from the first and last bisector. The
lower bound follows.

The upper bound follows from the triangle inequality, and the fact
that each Farey edge has length bounded by $2 \sqrt{ 30}\,
\pi^\frac{3}{4}$ by Theorem~\ref{farey edge}.
\end{proof}

In the language of the introduction, if $p/q$ has continued fraction
expansion 
$$\frac{p}{q} = [a_1, a_2, \ldots, a_n]$$
then $p/q$ has distance $n$ from $0$ in the Farey graph $\mathbb{F}$;
we say $p/q$ has {\em Farey depth} $n$,
or $\depth_{\mathbb{F}}
(p/q) = n$.

Then we have
$$  \frac{ \calV_3 \, \depth_{\mathbb{F}}(p/q)}{3 \sqrt{\pi/2}}   \le
d_{\overline{\rm WP}}\left(0,p/q\right) \le 
2 \sqrt{ 30}\, \pi^\frac{3}{4} 
\, \depth_{\mathbb{F}}(p/q) .$$

\bold{Remark.} We remark that genus independent upper bounds are
obtained in \cite{Cavendish:Parlier:diameter} on the extended
Weil-Petersson distance between maximally noded surfaces in terms of
the {\em cubical pants graph}, a modification of the usual pants graph
obtained by adding standard Euclidean $n$-cubes corresponding to
commuting families of elementary moves.  Explicit constants can be
given here in terms of the bounds on the length $\ell_{\rm WP}(I)$.
It is interesting to imagine how one might attempt lower bounds
without making use of the separation properties present in the
one-holed-torus and four-holed-sphere cases.

\bibliographystyle{math} 
\bibliography{math}

\begin{thebibliography}{GMM}

\bibitem[Ag]{Agol:tame}
I.~Agol.
\newblock {Tameness of hyperbolic 3-manifolds}.
\newblock {\em Preprint, {\tt arXiv:mathGT/0405568}} (2004).

\bibitem[ALM]{Agol:Leininger:Margalit:pseudo}
I.~Agol, C.~Leininger, and D.~Margalit.
\newblock {Pseudo-Anosov Stretch Factors and Homology of Mapping Tori}.
\newblock {\em arxiv:1409.6983}.

\bibitem[Ah1]{Ahlfors:simplicial}
L.~Ahlfors.
\newblock {An extension of Schwarz's lemma}.
\newblock {\em Trans. Amer. Math. Soc.} {\bf 43}(1938), 359--364.

\bibitem[Ah2]{Ahlfors:curvature}
L.~Ahlfors.
\newblock {Curvature properties of Teichm\"uller space}.
\newblock {\em J. Analyse Math.} {\bf 9}(1961), 161--176.

\bibitem[Ah3]{Ahlfors:visual}
L.~Ahlfors.
\newblock {Invariant operators and integral representations in hyperbolic
  space}.
\newblock {\em Math. Scan.} {\bf 36}(1975), 27--43.

\bibitem[Bon]{Bonahon:tame}
F.~Bonahon.
\newblock {Bouts des vari\'et\'es hyperboliques de dimension 3}.
\newblock {\em Annals of Math.} {\bf 124}(1986), 71--158.

\bibitem[Br1]{Brock:pdwp}
J.~Brock.
\newblock {Pants decompositions and the {W}eil-{P}etersson metric}.
\newblock In {\em Complex manifolds and hyperbolic geometry (Guanajuato,
  2001)}, volume 311 of {\em Contemp. Math.}, pages 27--40. Amer. Math. Soc.,
  Providence, RI, 2002.

\bibitem[Br2]{Brock:wp}
J.~Brock.
\newblock {The Weil-Petersson metric and volumes of 3-dimensional hyperbolic
  convex cores}.
\newblock {\em J. Amer. Math. Soc.} {\bf 16}(2003), 495--535.

\bibitem[Br3]{Brock:3ms1}
J.~Brock.
\newblock {Weil-Petersson translation distance and volumes of mapping tori}.
\newblock {\em Comm. Anal. Geom.} {\bf 11}(2003), 987--999.

\bibitem[BB1]{Brock:Bromberg:density}
J.~Brock and K.~Bromberg.
\newblock {On the density of geometrically finite Kleinian groups}.
\newblock {\em Acta Math.} {\bf 192}(2004), 33--93.

\bibitem[BB2]{Brock:Bromberg:inflexible}
J.~Brock and K.~Bromberg.
\newblock {Geometric inflexibility and 3-manifolds that fiber over the circle}.
\newblock {\em Journal of Topology} {\bf 4}(2011), 1--38.

\bibitem[BCM]{Brock:Canary:Minsky:elc}
J.~Brock, R.~Canary, and Y.~Minsky.
\newblock {The classification of Kleinian surface groups, II: the ending
  lamination conjecture}.
\newblock {\em Ann. of Math. (2)} {\bf 176}(2012), 1--149.

\bibitem[BM]{Brock:Minsky:spectrum}
J.~Brock and Y.~Minsky.
\newblock {Extended pseudo-Anosov maps and the Weil-Petersson length spectrum}.
\newblock {\em in preparation}.

\bibitem[BS]{Brock:Souto:exact}
J.~Brock and J.~Souto.
\newblock {On geometric invariants of mapping classes}.
\newblock {\em in preparation}.

\bibitem[Brm]{Bromberg:bers}
K.~Bromberg.
\newblock {Projective structures with degenerate holonomy and the Bers density
  conjecture}.
\newblock {\em Annals of Math.} {\bf 166}(2007), 77--93.

\bibitem[CG]{Calegari:Gabai:tame}
D.~Calegari and D.~Gabai.
\newblock {Shrinkwrapping and the taming of hyperbolic 3-manifolds}.
\newblock {\em J. AMS} {\bf 19}(2006), 385--446.

\bibitem[CT]{Cannon:Thurston:peano}
J.~W. Cannon and W.~P. Thurston.
\newblock {Group invariant Peano curves}.
\newblock {\em Geometry and Topology} {\bf 11}(2007), 1315--1355.

\bibitem[CM]{Cao:Meyerhoff:minvol}
C.~Cao and R.~Meyerhoff.
\newblock {The orientable cusped hyperbolic {$3$}-manifolds of minimum volume}.
\newblock {\em Invent. Math.} {\bf 146}(2001), 451--478.

\bibitem[CP]{Cavendish:Parlier:diameter}
W.~Cavendish and H.~Parlier.
\newblock {Growth of the {W}eil-{P}etersson diameter of moduli space}.
\newblock {\em Duke Math. J.} {\bf 161}(2012), 139--171.

\bibitem[Chu]{Chu:noncompleteness}
Tienchen Chu.
\newblock {The {W}eil-{P}etersson metric in the moduli space}.
\newblock {\em Chinese J. Math.} {\bf 4}(1976), 29--51.

\bibitem[DW]{Daskalopoulos:Wentworth:mcg}
G.~Daskolopoulos and R.~Wentworth.
\newblock {Classification of Weil-Petersson isometries}.
\newblock {\em Amer. J. Math.} {\bf 125}(2003), 941--975.

\bibitem[FLM]{Farb:Leininger:Margalit:volume}
B.~Farb, C.~Leininger, and D.~Margalit.
\newblock {Small dilatation pseudo-{A}nosov homeomorphisms and 3-manifolds}.
\newblock {\em Adv. Math.} {\bf 228}(2011), 1466--1502.

\bibitem[GMM]{Gabai:Meyerhoff:Milley}
D.~Gabai, R.~Meyerhoff, and P.~Milley.
\newblock {Minimum volume cusped hyperbolic three-manifolds}.
\newblock {\em J. Amer. Math. Soc.} {\bf 22}(2009), 1157--1215.

\bibitem[HK]{Hodgson:Kerckhoff:rigidity}
C.~Hodgson and S.~Kerckhoff.
\newblock {Rigidity of hyperbolic cone-manifolds and hyperbolic Dehn surgery}.
\newblock {\em J. Diff. Geom.} {\bf 48}(1998), 1--59.

\bibitem[KM]{Kojima:Mcshane:volumes}
S.~Kojima and G.~Mcshane.
\newblock {Normalized entropy versus volume for pseudo-Anosovs}.

\bibitem[Lin]{Linch:wp}
Michele Linch.
\newblock {A comparison of metrics on {T}eichm\"uller space}.
\newblock {\em Proc. Amer. Math. Soc.} {\bf 43}(1974), 349--352.

\bibitem[MM]{Manin:Marcolli}
Y.~Manin and M.~Marcolli.
\newblock {Holography principle and arithmetic of algebraic curves}.
\newblock {\em Adv. Theor. Math. Phys.} {\bf 5}(2001), 617--650.

\bibitem[Mc]{McMullen:book:RTM}
C.~McMullen.
\newblock {\em Renormalization and 3-Manifolds Which Fiber Over the Circle}.
\newblock Annals of Math. Studies 142, Princeton University Press, 1996.

\bibitem[Ot]{Otal:book:fibered}
J.~P. Otal.
\newblock {\em Le th\'eor\`eme d'hyperbolisation pour les vari\'et\'es
  fibr\'ees de dimension trois}.
\newblock Ast\'erisque, 1996.

\bibitem[Rei]{Reimann:visual}
H.M. Reimann.
\newblock {Invariant extension of quasiconformal deformations}.
\newblock {\em Ann. Acad. Sci. Fen.} {\bf 10}(1985), 477--492.

\bibitem[Schlk]{Schlenker:volume}
J.-M. Schlenker.
\newblock {The renormalized volume and the volume of the convex core of
  quasifuchsian manifolds}.
\newblock {\em Math. Res. Let.} {\bf 20}(2013), 773--786.

\bibitem[ST]{Schumacher:Trapani:conical}
G.~Schumacher and S.~Trapani.
\newblock {Weil-Petersson geometry for families of hyperbolic conical Riemann
  surfaces}.
\newblock {\em Michigan Math. J.} {\bf 60}(2011), 3--33.

\bibitem[Th1]{Thurston:book:GTTM}
W.~P. Thurston.
\newblock {\em Geometry and Topology of Three-Manifolds}.
\newblock Princeton lecture notes, 1979.

\bibitem[Th2]{Thurston:hype2}
W.~P. Thurston.
\newblock {Hyperbolic structures on 3-manifolds II: Surface groups and
  3-manifolds which fiber over the circle}.
\newblock {\em Preprint, {\tt arXiv:math.GT/9801045}} (1986).

\bibitem[Wol1]{Wolpert:noncompleteness}
S.~Wolpert.
\newblock {{Noncompleteness of the Weil-Petersson metric for Teichm\"uller
  space}}.
\newblock {\em Pacific J. Math.} {\bf 61}(1975), 573--577.

\bibitem[Wol2]{Wolpert:Nielsen}
S.~Wolpert.
\newblock {Geodesic length functions and the Nielsen problem}.
\newblock {\em J. Diff. Geom.} {\bf 25}(1987), 275--296.

\bibitem[Wol3]{Wolpert:compl}
S.~Wolpert.
\newblock {Geometry of the {W}eil-{P}etersson completion of {T}eichm\"uller
  space}.
\newblock In {\em Surveys in differential geometry, Vol.\ VIII (Boston, MA,
  2002)}, Surv. Differ. Geom., VIII, pages 357--393. Int. Press, Somerville,
  MA, 2003.

\bibitem[Wol4]{Wolpert:behavior}
S.~Wolpert.
\newblock {Behavior of geodesic-length functions on {T}eichm\"uller space}.
\newblock {\em J. Differential Geom.} {\bf 79}(2008), 277--334.

\bibitem[Yam]{Yamada:coxeter}
S.~Yamada.
\newblock {Weil-{P}etersson geometry of {T}eichm\"uller-{C}oxeter complex and
  its finite rank property}.
\newblock {\em Geom. Dedicata} {\bf 145}(2010), 43--63.

\end{thebibliography}

{\sc \small
 \bigskip

\noindent Brown University \bigskip

\noindent University of  Utah

}

\end{document}